\documentclass[12pt, reqno]{amsart}

\usepackage{amsfonts,amsmath}
\usepackage{amscd}
\usepackage{amssymb}

\allowdisplaybreaks

\usepackage{hyperref}

\date{\today}

\newtheorem{thm}{Theorem}[section]
\newtheorem{cor}[thm]{Corollary}

\theoremstyle{definition}

\theoremstyle{remark}
\newtheorem{rem}[thm]{Remark}

\numberwithin{equation}{section}

\newcommand{\R}{\mathbb R}

\newcommand{\He}{\mathbb H}

\newcommand{\C}{{\mathbb C}}

\renewcommand{\Im}{\operatorname{Im}}
\newcommand{\tr}{\operatorname{tr}}

\usepackage{color}

\title[Fourier transform on the Heisenberg group]
{ A scalar valued 
Fourier transform\\ for the Heisenberg group}

\author[ S. Thangavelu]{  Sundaram Thangavelu}

\address[S. Thangavelu]{Department of Mathematics\\
Indian Institute of Science\\
560 012 Bangalore, India}

\email{veluma@iisc.ac.in}

\date{}
 
\begin{document}

\maketitle

\begin{abstract}  We define a scalar valued Fourier transform for functions on the Heisenberg group and establish some of its basic properties like inversion formula, Plancherel theorem and Riemann-Lebesgue lemma. We also restate certain well known 
theorems for the group Fourier transform in terms of the new transform which we would like to call Strichartz Fourier transform.\\
 \end{abstract}

\begin{center} \it{Dedicated to the memory of Robert Strichartz\\}
\end{center}
\vskip0.25in
{\it I think it was during the summer of 1987 when my friend Der-Chen Chang brought a preprint   of Strichartz  to my attention.  In this work, which was published in 1989 \cite{RS1}, Strichartz developed harmonic analysis as spectral theory of Laplacians. In the context of the Heisenberg group $ \He^n$, the relevant operator was the sublaplacian $\mathcal{L}$, and Strichartz developed the joint spectral theory of $ \mathcal{L} $ and $ T.$  I read this  paper with great zeal and learnt for the first time about special Hermite functions and their importance in the harmonic analysis on $ \He^n.$  This point of view was further developed in another influential  paper \cite{RS2} and both works have played a vital role in all my works related to the Heisenberg group. During the academic year 1991-92, as a visitor to Cornell University, I was fortunate to have the opportunity to discuss mathematics with Strichartz. Needless to say, I greatly benefitted by interacting with him and I am eternally grateful for his kindness and the genuine interest he  showed in my works. As one can see, this article draws quite a lot from the two papers mentioned above.}

\section{Introduction}   Fourier transform of a function $ f $ on a locally compact Lie group $ G $ is defined  as a function on the unitary dual $ \widehat{G} $ consisting of the equivalence classes of irreducible unitary representations of $ G.$ Given a representative $ \pi $ of an element from $ \widehat{G} $  we define
\begin{equation}  \hat{f}(\pi) = \int_{G} f(g) \pi(g) dg \end{equation}
where $ dg $ is the left invariant Haar measure on $ G .$  It is also customary to use the notation $ \pi(f) $ instead of $ \hat{f}(\pi) $ especially when the group is non-abelian. Note that  for $ f \in L^1(G)$, $ \hat{f}(\pi) $ is a bounded linear operator on the Hilbert space $ \mathcal{H}_\pi $ on which the representation $ \pi $ is realised. When $G $ is abelian, all the irreducible unitary representations of $ G $ are one dimensional and hence we can think of them as homomorphisms of $ G $ into the circle group $ S^1 $ consisting of all complex numbers of absolute value one. They are known as unitary characters and for any given $ \chi \in \widehat{G} $ the Fourier transform
\begin{equation}
  \hat{f}(\chi) = \int_G  f(g) \chi(g)  dg 
  \end{equation}
   is a scalar. We can also turn $ \widehat{G} $ into a multiplicative group and hence $ \hat{f} $ becomes a function on this dual group. This allows us to treat Fourier transform on locally compact abelian groups  on par with the well understood Euclidean Fourier transform on $ \R^n.$\\
   
 When the group $ G $ is non abelian, then it is no longer true that all the elements of $ \widehat{G} $ are one dimensional.   If the group is compact, then $ \widehat{G} $ consists only of finite dimensional representations and hence  Fourier transform becomes a matrix valued function which is still manageable. When the group is non compact and non abelian, we need to deal with infinite dimensional representations and hence the Fourier transform becomes operator-valued. As we are mainly interested  in studying  the Fourier transform on the Heisenberg group, let us leave the generality and specialise to the case in hand.\\
 
 The simplest example of a non abelian nilpotent Lie group  is provided by the Heisenberg group  $ \He^n $ whose underlying manifold is $ \C^n \times \R .$ 
The representation theory of $ \He^n $  is very simple and its unitary dual can be described  explicitly. There are certain one dimensional representations of $ \He^n $ which are characters, parametrised by $ \C^n $, but we discard them as they do not play any role in the Plancherel theorem. We let $ \widehat{\He^n} $ to stand for the set of all infinite dimensional irreducible unitary representations of $ \He^n.$ By a theorem of Stone and von Neumann, the members of $ \widehat{\He^n} $ are parametrised by non zero real numbers $ \R^\ast.$ Thus to each $ \lambda \in \R^\ast $ we have an irreducible unitary representation $ \pi_\lambda $ realised on the same Hilbert space $ L^2(\R^n).$ Thus the Fourier transform of $ f \in L^1(\He^n) $ is the operator valued function
 $$  \hat{f}(\lambda) = \int_{\He^n} f(g) \pi_\lambda(g) dg $$
 defined on $ \R^\ast.$ This operator valued  Fourier transform satisfies all the standard properties of a Fourier transform. Hence for suitable functions we have an inversion formula and for $ f \in L^1 \cap L^2(\He^n) $ we have the Plancherel formula:
 \begin{equation}
 \int_{\He^n} |f(g)|^2 \, dg = (2\pi)^{-n-1} \int_{-\infty}^\infty \| \hat{f}(\lambda)\|_{HS}^2 \, |\lambda|^n \, d\lambda.
 \end{equation}
This allows us to extend the Fourier transform to the whole of $ L^2(\He^n) $ as a unitary operator onto $ L^2(\R^\ast, \mathcal{S}_2, d\mu),$ the $ L^2 $ space of Hilbert-Schmidt operator valued functions on $ \R^\ast$ taken with respect to the measure $ d\mu(\lambda) = (2\pi)^{-n-1} |\lambda|^n d\lambda.$\\

Though the Fourier transform defined above shares many features  with its abelian counterparts, it is unwieldy and not suitable for studying several standard problems in harmonic analysis. For example, there are no simple descriptions of the images either of the Schwartz space $ \mathcal{S}(\He^n) $ or $ C_0^\infty(\He^n) $ under the Fourier transform. Consequently, it is not clear how to define Fourier transforms of distributions on the Heisenberg group. Though there are versions of Paley-Wiener theorems for the so called Fourier-Weyl transform, the results are not very satisfactory, see \cite{ST}, \cite{ST2}.\\

Recently the authors of the paper \cite{BCD} have attempted with certain degree of success to define a scalar valued Fourier transform on $ \He^n.$  The idea is very simple: as $ \hat{f}(\lambda)$ for $ f \in L^1(\He^n) $ is a bounded linear operator on $ L^2(\R^n),$ it is natural to fix an orthonormal basis $ \Phi_\alpha^\lambda,\, \alpha \in \mathbb{N}^n $ for $ L^2(\R^n) $ and consider the map $ f \rightarrow  \langle \hat{f}(\lambda)\Phi_\alpha^\lambda, 
\Phi_\beta^\lambda \rangle $ as a candidate for a scalar valued Fourier transform. In \cite{BCD} the authors have introduced a metric $ d $ on the set $ \widetilde{\He}^n = \mathbb{N}^n \times \mathbb{N}^n \times \R^\ast $ so that the above map is continuous from $ L^1(\He^n) $ into $ \widehat{\He}^n,$ the completion of $ \widetilde{\He}^n $ in the metric $ d.$\\

In this article we propose the following definition of  a scalar valued Fourier transform on $ \He^n $ that shares several properties with the Helgason Fourier transform on noncompact rank one Riemannian symmetric spaces.  In order to define this Fourier transform, which we call Strichartz Fourier transform, we need to set up some notations.  For any $ \delta \geq -1/2 $ we let $ L_k^\delta(t), k \in \mathbb{N}, t \geq 0 $ stand for the Laguerre polynomials of type $ \delta.$  For any $ \lambda \in \R^\ast,$ we define the Laguerre functions by
$$ \varphi_{k,\lambda}^{n-1}(z) = L_k^{n-1}(\frac{1}{2}|\lambda||z|^2) e^{-\frac{1}{4}|\lambda| |z|^2}, \, \, z \in \C^n. $$
Then  $ e_{k,\lambda}^{n-1}(z,t) = e^{i\lambda t}  \varphi_{k,\lambda}^{n-1}(z) $ are joint eigenfunctions of the sublaplacian $ \mathcal{L} $ and $  T = \frac{\partial}{\partial t}$: 
$$ \mathcal{L} e_{k,\lambda}^{n-1}(z,t) = (2k+n)|\lambda|\, e_{k,\lambda}^{n-1}(z,t),\,\, -iT e_{k,\lambda}^{n-1}(z,t) = \lambda \, e_{k,\lambda}^{n-1}(z,t).$$
Let $ \Omega $ stand for the Heisenberg fan which is the union of the rays $ R_k = \{ (\lambda, \tau) \in \R^2: \tau = (2k+n)|\lambda| \} $ for $ k=0,1,2,... $ and the limiting ray $ R_\infty = \{ (0,\tau): \tau \geq 0 \}.$  For each $ a = (\lambda, (2k+n)|\lambda|) \in R_k $ we use the notation $ e_a(z,t) $ in place of  $e_{k,\lambda}^{n-1}(z,t).$  For any $ f \in L^1 \cap L^2(\He^n),$ we define its Strichartz Fourier transform $ \widehat{f}(a,z) $ on $ \Omega \times \C^n $ as follows. For $ a \in R_k, z \in \C^n ,
\lambda \neq 0$ we define
\begin{equation}\label{def-k} \widehat{f}(a,z) = \int_{\He^n} f(w,s) e_a((w,s)^{-1}(z,0)) dw \,ds.\end{equation}
Here $ (w,s)^{-1} =(-w,-s) $ and the group law in $\He^n $ is given in Section 2. For  $ a =(0,\tau) $ coming from the limiting ray $ R_\infty $ we set 
\begin{equation}\label{def-infty} \widehat{f}(0,\tau,z) = (n-1)! 2^{n-1}\,  \int_{\He^n} f(w,s) \frac{J_{n-1}(\sqrt{\tau} |z-w|)}{(\sqrt{\tau} |z-w|)^{n-1}} dw \, ds\end{equation}
where $ J_{n-1} $ is the Bessel function of order $(n-1).$ As a subset of $ \R^2 , \Omega $ inherits the Euclidean metric and topology. We define the normalised Strichartz Fourier transform $  \widetilde{f}(a,z) $ 
by  
$$  \widetilde{f}(a,z) = \frac{k!(n-1)!}{(k+n-1)!}  \widehat{f}(a,z),\,\,\, a \in R_k,\,\,\,    \widetilde{f}(0,\tau,z) = \widehat{f}(0,\tau,z) .$$  For this transform we can prove  the following analogue of Riemann-Lebesgue lemma.\\

\begin{thm}\label{riem-leb} For any  $ f \in L^1(\He^n),$ the normalised Strichartz Fourier transform $  \widetilde{f}(a,z) $ 
is uniformly  continuous   on $ \Omega$  for any $ z \in \C^n $ fixed. Moreover,  $  \widetilde{f}(a,z) $  vanishes at infinity, i.e. $ \widetilde{f}(a,z) \rightarrow 0 $ as $ |a| \rightarrow \infty $ for each $ z \in \C^n $ fixed.
\end{thm}

On  the Heisenberg fan  $ \Omega $ we consider the measure $ \nu $ which is defined by
$$ \int_\Omega \varphi(a) d\nu(a) = (2\pi)^{-2n-1} \,\int_{-\infty}^\infty \Big( \sum_{k=0}^\infty   \varphi(\lambda, (2k+n)|\lambda|) \Big)  |\lambda|^{2n} d\lambda.$$
We can now state the inversion formula and the Plancherel theorem for our Strichartz Fourier transform.

\begin{thm}\label{inv-plan} For any Schwartz class function $ f $ on $ \He^n $ we have the following inversion formula for the Strichartz Fourier transform:
\begin{equation}\label{inverse} f(z,t) = \int_\Omega \int_{\C^n}  \widehat{f}(a,w) e_a((-w,0)(z,t))  dw \, d\nu(a).
\end{equation}
 Moreover, for any $ f \in L^1 \cap L^2(\He^n) $ we have the Plancherel formula
\begin{equation}\label{plancherel} \int_{\He^n} |f(z,t)|^2 dz\, dt =   \int_\Omega \int_{\C^n}  |\widehat{f}(a,w)|^2 dw \,d\nu(a).
\end{equation}
\end{thm}

It is not true that the Strichartz Fourier transform takes $ L^2(\He^n) $ onto $ L^2(\Omega \times \C^n, d\nu\, dw).$ This is because $ \widehat{f}(a,z) $ has to satisfy a  necessary condition (see  \eqref{necessary}) which we write as
$$ (2\pi)^{-n} |\lambda|^n\,   \varphi_{k,\lambda}^{n-1} \ast_\lambda  \widehat{f}(a,\cdot)(z)=   \widehat{f}(a,z),\,\,\, a \in R_k.$$
If we let $ L^2_0(\Omega \times \C^n, d\nu\, dw) $ stand for the subspace of $ L^2(\Omega \times \C^n, d\nu\, dw)$ consisting of functions $ F(a,z) $ satisfying the above condition, we can prove the following result.

\begin{thm}\label{unitary} The Strichartz Fourier transform initially defined on  $L^1 \cap L^2(\He^n) $  can be extended as a unitary operator from $ L^2(\He^n) $ onto $ L^2_0(\Omega \times \C^n, d\nu\, dw).$
\end{thm}

\begin{rem} For $ f \in L^2(\He^n) $ the function $ f^\lambda $ is defined only for almost every $ \lambda \in \R^\ast $ and hence  we may not be able to define $ \widehat{f}(a,z) $ at every $ a \in \Omega.$ In particular, $ \widehat{f}(0,\tau,z) $ need not be defined. As the limiting ray $ R_\infty $ has zero Plancherel measure, this does not create any problem.
\end{rem}

As  a consequence of the estimates satisfied by the Laguerre functions $ \varphi_{k,\lambda}^{n-1}(z) $ and the Bessel function $ J_{n-1}(t),$  the  Strichartz Fourier transform of any $ f \in L^1(\He^n) $ satisfies the estimate
$$ \sup_{(a,z) \in \Omega \times \C^n} |\widetilde{f}(a,z)| \leq c_{n,1}\, \|f\|_1 .$$
Restating the Plancherel theorem theorem in terms of  $ \widetilde{f} $  and interpolating  with change of measures we can prove the following Hausdorff-Young inequality.

\begin{thm}\label{haus-young} For $ 1 \leq p \leq 2 $ there is a measure $ \nu_p $ on $ \Omega $ such that for all  $ f \in L^p(\He^n)$ we have 
$$ \Big( \int_\Omega \int_{\C^n}  |\widetilde{f}(a,w)|^{p^\prime} dw \,d\nu_p(a) \Big)^{1/p^\prime} \leq  c_{n,p}\, \Big(\int_{\He^n} |f(z,t)|^p dz\, dt \Big)^{1/p}.$$\\
\end{thm} 

Let us compare the definition \eqref{def-k} with the Euclidean Fourier transform and the Helgason Fourier transform on noncompact rank one Riemannian symmetric spaces $ X = G/K.$  For functions $ f $ on $ \R^n $ we can write the Fourier transform in the form
\begin{equation}\label{e-fouri} \widehat{f}(\lambda,\omega) = (2\pi)^{-n/2} \int_{\R^n}  f(x) e^{-i\lambda x \cdot \omega} dx,\,\,\, \lambda \in \R, \omega \in S^{n-1}.\end{equation}
Note that the functions $e_{\lambda,\omega}(x) = e^{-i\lambda x \cdot \omega}$ are eigenfunctions of the Laplacian $ \Delta $ and $ \R^n$ and  the inversion formula which reads as 
\begin{equation}\label{e-fouri-in} f(x) = (2\pi)^{-n/2} \int_{-\infty}^\infty \int_{S^{n-1}} \widehat{f}(\lambda,\omega) e_{\lambda,\omega}(x) |\lambda|^{n-1} d\omega\, d\lambda \end{equation}
expresses $ f $ as a superposition of the eigenfunctions $ e_{\lambda,\omega}.$ As $ e_a(z,t) $ are eigenfunctions of $ \mathcal{L} $  the similarity between \eqref{inverse} and \eqref{e-fouri-in} is clear. It is even more illuminating to compare our Fourier transform with the Helgason Fourier transform.\\

Let $ X = G/K $ be a rank one symmetric space of noncompact type. The Helgason Fourier transform  $ \widetilde{f}(\lambda,b) $ of a function $ f $ on $ G/K $ is given by
\begin{equation}\label{h-fouri} 
\widetilde{f}(\lambda,b) = \int_{G/K} f(x) e^{(-i\lambda+\rho)A(x,b)} dx 
\end{equation}
where $ \lambda \in \C $ and $ b \in K/M,$ see Helgason \cite{SH1,SH2} for the notations. The inversion formula is given by
\begin{equation}\label{h-fouri-in}
f(x) = c\, \int_{-\infty}^\infty \int_{K/M} \widetilde{f}(\lambda,b)e^{(i\lambda+\rho)A(x,b)} |c(\lambda)|^{-2} db\, d\lambda 
\end{equation}
where $ db $ is the normalised measure on $ K/M $ and $ c(\lambda) $ is the Harish-Chandra $c$-function. Here again, the functions $ e_{\lambda,b}(x) = e^{(i\lambda+\rho)A(x,b)} $ are eigenfunctions of the Laplace-Beltrami operator $ \Delta_X $ on the symmetric space $ X$ bringing out the analogy between \eqref{def-k} and \eqref{h-fouri}.\\

Since functions $ f $ on $ G/K $ can be considered as right $K$-invariant functions on the Lie group $ G $ we can consider the group Fourier transform of $ f.$ For any $ \lambda \in \C $ if we let $ \pi_\lambda $ stand  for the spherical principal series representations realized on $ L^2(K/M),$  each of these representations has  a unique $ K$-fixed vector $ Y_0 \in L^2(K/M) $ which is the constant function $ Y_0(b) =1.$  The Helgason Fourier transform of $ f $ is then related to the group Fourier 
transform via the equation
\begin{equation}\label{group-v-helga} \widetilde{f}(\lambda,b) = \pi_\lambda(f)Y_0(b),\,\,\, \pi_\lambda(f) = \int_G f(g) \pi_\lambda(g) dg.
\end{equation}
We will show that there is a similar relation in the case of Strichartz Fourier transform. The Heisenberg group $ \He^n $ can be considered as a subgroup of a bigger group $ G_n:=\mathbb{H}^n \ltimes U(n)$ known as the Heisenberg motion group, so that functions on $ \He^n $ are precisely the right $U(n)$-invariant functions on $ G_n.$ For each $ a = (\lambda, (2k+n)|\lambda|) \in R_k $ there is a class one representation $ \rho_k^\lambda $ of $ G_n $ so that
\begin{equation}\label{group-v-stri} \widehat{f}(a,z) = \rho_k^\lambda(f)e_a(z,0),\,\,\, \rho_k^\lambda(f) = \int_{G_n} f(g) \rho_k^\lambda(g) dg.
\end{equation}
We will demonstrate that the Strichartz Fourier transform shares several other properties with the Helgason Fourier transform.\\

We conclude this introduction with a brief description of the plan of the paper. In Section 2,  after recalling the representation theory of $ \He^n $ and the Heisenberg group motion group and the spectral theory of the sublaplacian, we define the Strichartz Fourier transform and  prove all its basic properties. In Section 3 we explore the connection between the Strichartz Fourier transform and the Gelfand transform on the commutative Banach algebra consisting of $ L^1(\He^n) $ functions which are radial in the $z$ variable. We  prove an analogue of  Hecke-Bochner formula for the Strichartz Fourier transform and use it to characterize the image of the Schwartz space under the Strichartz Fourier transform. In Section 4 we describe the connection between the operator valued group Fourier transform and the scalar valued Strichartz Fourier transform and restate the well known Hardy and Ingham theorems which are originally stated in terms of the former, in terms of the new transform.

\section{Fourier transforms on the Heisenberg group}	

\subsection{ Schr\"odinger representations and the group Fourier transform}Let $\mathbb{H}^n:=\mathbb{C}^n\times\mathbb{R}$ be the $(2n+1)$- dimensional Heisenberg group with the group law
	 
	$$(z, t)(w, s):=\big(z+w, t+s+\frac{1}{2}\Im(z.\bar{w})\big),\ \forall (z,t),(w,s)\in \mathbb{H}^n.$$ This is a step two nilpotent Lie group where the Lebesgue measure $dzdt$ on $\mathbb{C}^n\times\mathbb{R}$ serves as the Haar measure. The representation theory of $\mathbb{H}^n$ is well-studied in the literature, see the monographs \cite{F}, \cite{MT} and \cite{ST2}. In order to define Fourier transform, we use the Schr\"odinger representations as described below.  
	
	For each non zero real number $ \lambda $ we have an infinite dimensional representation $ \pi_\lambda $ realised on the Hilbert space $ L^2( \R^n).$ These are explicitly given by
	$$ \pi_\lambda(z,t) \varphi(\xi) = e^{i\lambda t} e^{i\lambda(x \cdot \xi+ \frac{1}{2}x \cdot y)}\varphi(\xi+y),\,\,\,$$
	where $ z = x+iy $ and $ \varphi \in L^2(\R^n).$ These representations are known to be  unitary and irreducible. Moreover, by a theorem of Stone and Von-Neumann, (see e.g., \cite{F})  up to unitary equivalence these account for all the infinite dimensional irreducible unitary representations of $ \mathbb{H}^n $ which act as $e^{i\lambda t}I$ on the center. Also there is another class of finite dimensional irreducible representations. As they  do not contribute to the Plancherel measure  we will not describe them here.\\
	
	The Fourier transform of a function $ f \in L^1(\mathbb{H}^n) $ is the operator valued function   on the set of all nonzero reals $ \R^\ast $ given by
	$$ \hat{f}(\lambda) = \int_{\mathbb{H}^n} f(z,t) \pi_\lambda(z,t)  dz dt .$$  Note that $ \hat{f}(\lambda) $ is a bounded linear operator on $ L^2(\R^n).$ It is known that when $ f \in L^1 \cap L^2(\mathbb{H}^n) $ its Fourier transform  is actually a Hilbert-Schmidt operator and one has
	$$ \int_{\mathbb{H}^n} |f(z,t)|^2 dz dt = (2\pi)^{-(n+1)}\int_{-\infty}^\infty \|\widehat{f}(\lambda)\|_{HS}^2 |\lambda|^n d\lambda  $$
	where $\|.\|_{HS}$ denotes the Hilbert-Schmidt norm. 
	The above allows us to extend  the Fourier transform as a unitary operator between $ L^2(\mathbb{H}^n) $ and the Hilbert space of Hilbert-Schmidt operator valued functions  on $ \R $ which are square integrable with respect to the Plancherel measure  $ d\mu(\lambda) = (2\pi)^{-n-1} |\lambda|^n d\lambda.$ \\
	
	Let $ \mathcal{S}_2 $ stand for the Hilbert space of Hilbert-Schmidt operators on $ L^2(\R^n) $ equipped with the inner product $ (T,S) = \tr(S^\ast T).$ We then have the following Plancherel theorem:
	
	\begin{thm} The group Fourier transform is a unitary operator from $ L^2(\He^n) $ onto $ L^2(\R^\ast, \mathcal{S}_2, d\mu).$
	\end{thm}

	By  polarizing the Plancherel formula   we  obtain the Parseval identity: for $ f, g \in L^2(\He^n) $ 
	$$\int_{\He^n}f(z,t)\overline{g(z,t)}dzdt=\int_{-\infty}^{\infty}tr(\widehat{f}(\lambda)\widehat{g}(\lambda)^*)~d\mu(\lambda).$$ Also for suitable functions $f$ on $\He^n$ we have the following inversion formula
	$$f(z,t)=\int_{-\infty}^{\infty}tr(\pi_{\lambda}(z,t)^*\widehat{f}(\lambda))d\mu(\lambda).$$
	Now from the  definition of $\pi_{\lambda}$  it is easy to see that 
	$\widehat{f}(\lambda)=\int_{\C^n}f^{\lambda}(z)\pi_{\lambda}(z,0)dz $ where 
	$f^{\lambda}$ stands for the inverse Fourier transform of $f$ in the central variable:
	$$f^{\lambda}(z):=\int_{-\infty}^{\infty}e^{i\lambda t}f(z,t)dt.$$
	This suggests that for any $g \in L^1(\C^n) $, we consider the following   operator valued function  
		$$ W_{\lambda}(g):=\int_{\C^n}g(z)\pi_{\lambda}(z,0)dz.
		$$ With these notations we note that  $\hat{f}(\lambda)=W_{\lambda}(f^{\lambda}).$  These transforms are called the Weyl transforms. We have the following Plancherel formula for Weyl transform  (See \cite[(2.2.9), p.  49]{ST3})
		\begin{equation}
			\label{wplan}
			 \|W_{\lambda}(g)\|^2_{HS}|\lambda|^n=(2\pi)^n\|g\|_2^2, ~g\in L^2(\mathbb{C}^n).
		\end{equation}   
		Many results for the group Fourier transform are proved by studying the analogues for the Weyl transform.

\subsection{Joint spectral theory of $ \mathcal{L} $ and $ T$} The Heisenberg Lie algebra $ \mathfrak{h}_n $ has a basis consisting of the left invariant  vector fields
$$ X_j = \frac{\partial}{\partial x_j}+ \frac{1}{2} y_j \frac{\partial}{\partial t},\,\,X_j = \frac{\partial}{\partial y_j}- \frac{1}{2} x_j \frac{\partial}{\partial t},\,\, T =  \frac{\partial}{\partial t},\,\, j =1,2,...,n.$$
The sublaplacian $ \mathcal{L} = -\sum_{j=1}^n \big( X_j^2+Y_j^2\big) $ plays the role of the Laplacian on $ \He^n.$ As $ \mathcal{L} $ commutes with $ T $ there is a well defined joint spectral theory of these two operators.  The functions $ e_{k,\lambda}^{n-1}(z,t) $ in the introduction are joint eigenfunctions of $ \mathcal{L} $ and $ T.$  The operators $ L_\lambda $ defined by the relation $ \mathcal{L}(f(z) e^{i\lambda t}) = e^{i\lambda t} L_\lambda f(z) $ are known as a special Hermite operators and they have explicit spectral decomposition which we describe now.\\

The convolution on $ \He^n $ gives rise to the so called $\lambda$-twisted convolutions by the relation
$$ (f\ast g)^\lambda(z) = \int_{\C^n} f^\lambda(z-w) g^\lambda(w) e^{\frac{i}{2}\lambda \Im(z \cdot \bar{w})} dw =: f^\lambda \ast_\lambda g^\lambda(z).$$
The Laguerre functions $ \varphi_{k,\lambda}^{n-1} $ are eigenfunctions of $ L_\lambda $ with eigenvalues $ (2k+n)|\lambda| $ and every $ g \in L^2(\C^n) $ has the $ L^2 $ convergent expansion
\begin{equation}\label{spl-her-exp}
g(z) = (2\pi)^{-n} \, |\lambda|^n \, \sum_{k=0}^\infty g \ast_\lambda \varphi_{k,\lambda}^{n-1}(z). 
\end{equation}
The eigenspaces of $ L_\lambda $ corresponding to the eigenvalues $ (2k+n)|\lambda| $ are infinite dimensional and an orthonormal basis is provided by 
$$ \Phi_{\alpha,\beta}^\lambda(z) = (2\pi)^{-n/2} (\pi_\lambda(z,0)\Phi_\alpha^\lambda, \Phi_\beta^\lambda),\,\,  \alpha, \beta \in \mathbb{N}^n, |\alpha| =k.$$ 
Here $ \Phi_\alpha^\lambda$ are the scaled Hermite functions which are eigenfunctions of the scaled Hermite operator $ H(\lambda) = -\Delta+\lambda^2 |x|^2 $ with eigenvalues $ (2k+n)|\lambda| .$ We refer to \cite{ST3} for the details. 
The functions $\Phi_{\alpha,\beta}^\lambda,$ known as special Hermite functions, form an orthonormal basis for $ L^2(\C^n) $ and \eqref{spl-her-exp} is the compact form of the special Hermite expansion of a function $ g.$\\

The projections $ (2\pi)^{-n} |\lambda|^n g \ast_\lambda \varphi_{k,\lambda}^{n-1} $ are mutually orthogonal and the Plancherel theorem associated to the expansion \eqref{spl-her-exp} reads as 
\begin{equation}\label{spl-her-plan}
\int_{\C^n} |g(z)|^2 dz  = (2\pi)^{-2n} \, |\lambda|^{2n}\, \sum_{k=0}^\infty \int_{\C^n} |g \ast_\lambda \varphi_{k,\lambda}^{n-1}(z)|^2 dz. 
\end{equation}
Weyl transform converts the twisted convolution into products and we have  the formula
\begin{equation}\label{projection}
(2\pi)^{-n} |\lambda|^n W_\lambda( g \ast_\lambda \varphi_{k,\lambda}^{n-1}) = W_\lambda(g) P_k(\lambda).\,\, 
\end{equation}
Here $ P_k(\lambda) $ is the orthogonal projection of $L^2(\R^n) $ onto the $k$-th eigenspace of the Hermite operator $ H(\lambda) $ spanned by $ \Phi_\alpha^\lambda,  |\alpha|=k.$

\subsection{Heisenberg motion group and some class one representations} Let $U(n)$ denote the group of all unitary matrices of order $n$. This acts on $\mathbb{H}^n$ by the automorphisms $$\sigma.(z,t)=(
\sigma z,t),~\sigma\in U(n).$$  We consider the semi-direct product of $\mathbb{H}^n$ and $U(n)$, $G_n:=\mathbb{H}^n \ltimes U(n)$ which acts on $\mathbb{H}^n$ by  $$(z,t,\sigma).(w,s)=\big(z+\sigma w, t+s+\frac12 \Im( z \cdot \overline{\sigma w})\big)$$  whence  the group law in $G_n$ is given by
$$(z,t,\sigma).(w,s,\tau)=\big(z+\sigma w, t+s+\frac12 \Im( z \cdot \overline{\sigma w}), \sigma \tau \big).$$ The group $G_n$ is called the \textit{Heisenberg motion group} which contains $\mathbb{H}^n
$ and $U(n)$ as subgroups. Also $\mathbb{H}^n$ can be identified with the quotient group $G_n/U(n)$. As a matter of fact, functions on $\mathbb{H}^n$ can be viewed as right $U(n)$ invariant functions on $G_n.$ The Haar measure on $G_n$ is given by $d\sigma\,dz\,dt$ where $d\sigma$ denotes the normalised Haar measure on $U(n)$. To bring out the connection between the group Fourier transform on $\mathbb{H}^n$ and the Heisenberg motion group, we need to describe a family of class one representations of $G_n.$ We start with recalling the definition of such representations. \\

 Let $G$ be a locally compact topological group and $K$ be a compact subgroup of $G$. Suppose $\pi$ is a representation of $G$ realised on the Hilbert space $H$. Let $H_{K}$ denote the set of all $K$-fixed vectors given by 
 $$H_{K}:=\{v\in H: \pi(k)v=v,~\forall k\in K\}.$$ It can be easily checked that $H_K$ is a subspace of $H$.
  We say that $\pi$ is a class one representation of the pair $(G,K)$ if $H_{K}\neq\{0\}.$ Moreover, when $(G, K)$ is a Gelfand pair it is well-known that $\dim H_K=1.$ In the following, we describe certain family of class one representation for the Gelfand pair $(G_n, U(n))$. For that we need to set up some more notations. \\
  
  For $\alpha\in \mathbb{N}^n$ and $\lambda\neq 0$ let $\Phi^{\lambda}_{\alpha}(x):=|\lambda|^{n/4}\Phi_{\alpha}(\sqrt{|\lambda|}x),~x\in \mathbb{R}^n$ where $\Phi_{\alpha}$ denote the normalised Hermite functions on $\mathbb{R}^n.$ We know that for each $\lambda\neq 0$, $\{\Phi_{\alpha}^{\lambda}: \alpha\in\mathbb{N}^n\}$ forms an orthonormal basis for $L^2(\mathbb{R}^n)$. Suppose 
  $$E^{\lambda}_{\alpha,\beta}(z,t):=(\pi_{\lambda}(z,t)\Phi_{\alpha}^{\lambda}, \Phi_{\beta}^{\lambda}),~(z,t)\in \mathbb{H}^n$$ denotes the matrix coefficients of the Schr\"odinger representation $\pi_{\lambda}$ of $\mathbb{H}^n.$  For each $\lambda\neq0$ and $k\in\mathbb{N}$, we consider the Hilbert space $\mathcal{H}^{\lambda}_k$ spanned by $\{E^{\lambda}_{\alpha,\beta}:\alpha, \beta\in\mathbb{N}^n, |\beta|=k\}$
and equipped with the inner product
$$(f,g)_{\mathcal{H}^{\lambda}_k}:=(2\pi)^{-n}|\lambda|^n\int_{\C^n}f(z,0)\overline{g(z,0)}dz.$$ We define a representation $\rho_k^{\lambda}$ of $G_n$ realised on $\mathcal{H}^{\lambda}_k$ by the prescription
$$\rho_k^{\lambda}(z,t,\sigma)\varphi(w,s):=\varphi((z,t,\sigma)^{-1}(w,s)),~(w,s)\in \mathbb{H}^n.$$ It is well-known that $\rho_k^{\lambda}$ is an irreducible unitary representation of $G_n$ for all $\lambda\neq0$ and $k\in\mathbb{N}.$ Also for $\lambda\neq 0$ and $k\in \mathbb{N}$ we consider the function $e_{k,\lambda}^{n-1}$ on $\mathbb{H}^n$ defined by 
$$e_{k,\lambda}^{n-1}(z,t)= \sum_{|\alpha|=k}(\pi_{\lambda}(z,t)\Phi_{\alpha}^{\lambda}, \Phi_{\alpha}^{\lambda}).$$ It is known that the above function can be expressed in terms of Laguerre functions as follows (See \cite[p. 52]{ST3})
$$e_{k,\lambda}^{n-1}(z,t)= e^{i\lambda t}\varphi_{k,\lambda}^{n-1}(z).$$
It can be checked that $e_{k,\lambda}^{n-1}$ is a $U(n)$-fixed vector corresponding to the representation $\rho_k^{\lambda}$ and hence $\rho_k^{\lambda}$ is a class one representation of the pair  $(G_n, U(n))$. Moreover,  $(G_n, U(n))$ being a Gelfand pair, $e_{k,\lambda}^{n-1}$ is unique up to scalar multiple.\\

The representations $ \rho_k^\lambda $ when restricted to $ \He^n $ are not irreducible but split into finitely many irreducible unitary representations each one being equivalent to $ \pi_\lambda.$ Given $f\in L^1(\mathbb{H}^n)$, considering it as an $U(n)$-invariant function on $G_n$, we associate an operator valued function $\rho_k^{\lambda}(f)$ acting on $\mathcal{H}^{\lambda}_k$ defined by 
$$\rho_k^{\lambda}(f):=\int_{G_n}f(z,t)\rho_k^{\lambda}(z,t,\sigma)d\sigma\,dz\,dt.$$ 
Now since $\rho_k^{\lambda}$ is unitary, it can be easily checked that $\rho_k^{\lambda}(f)$ is a bounded operator and the operator norm is bounded above by $\|f\|_1.$  From the definition of $\rho_k^{\lambda}$ the following can be easily checked: 
\begin{equation}\label{spl-her-pro} \rho^{\lambda}_{k}(f)e^{n-1}_{k,\lambda}(z,t)= e^{i \lambda t} f^{-\lambda} \ast_{-\lambda} \varphi_{k,\lambda}^{n-1}(z).
\end{equation} 
  This leads to the  Plancherel  formula, proved in \cite[Proposition 2.1]{RRT}, for the representations $ \rho_k^\lambda.$ 
\begin{equation}\label{norm-rho-k}\frac{(k+n-1)}{k! (n-1)!} \| \rho_k^\lambda(f)\|_{HS}^2  = (2\pi)^{-n} |\lambda|^n \,  \int_{\C^n} |f^{-\lambda}\ast_{-\lambda}\varphi^{n-1}_{k,\lambda}(z)  |^2 dz. \end{equation}
It is easy to see that $ \rho^{\lambda}_{k}(f)e^{n-1}_{k,\lambda}(z,t)$ are eigenfunctions of the sublaplacian  $\mathcal{L} $ with eigenvalues $ (2k+n)|\lambda| $ and $ f $ can be recovered  by the formula
\begin{equation}\label{inv}
f(z,t) = (2\pi)^{-n-1}\, \int_{-\infty}^\infty   e^{i\lambda t}\,  \big(\sum_{k=0}^\infty  \rho_k^\lambda(f)e^{n-1}_{k,\lambda}(z,0) \big) |\lambda|^n d\lambda.
\end{equation}
This is a consequence of the special Hermite expansion \eqref{spl-her-exp} applied to $ f^\lambda.$ We can thus view the above as the spectral decomposition of the sublaplacian.\\

\subsection{Strichartz  Fourier transform on the Heisenberg group} We propose the following definition as a scalar valued Fourier transform for functions on the Heisenberg group.  Let $ \Omega $ stand for the Heisenberg fan which is the union of the rays $ R_k = \{ (\lambda, \tau) \in \R^2: \tau = (2k+n)|\lambda| \} $ for $ k=0,1,2,... $ and the limiting ray $ R_\infty = \{ (0,\tau): \tau \geq 0 \}.$  For any $ f \in L^1 \cap L^2(\He^n),$ we define its Strichartz Fourier transform $ \widehat{f}(a,z) $ for $ a \in R_k, z \in \C^n $ by the relation
$$  \widehat{f}(a,z) =   f^{-\lambda} \ast_{-\lambda} \varphi_{k,\lambda}^{n-1}(z) ,\,\,\, a = (\lambda, (2k+n)|\lambda|).$$
In view of the relation \eqref{spl-her-pro}, namely $ \rho_k^\lambda(f)e_{k,\lambda}^{n-1}(z,t)= e^{i\lambda t} f^{-\lambda} \ast_{-\lambda} \varphi_{k,\lambda}^{n-1}(z)$ we see that
$$  \widehat{f}(a,z) = e^{-i\lambda t} \, \rho_k^\lambda(f)e_{k,\lambda}^{n-1}(z,t) =  \rho_k^\lambda(f)e_{k,\lambda}^{n-1}(z,0).$$
For  $ a =(0,\tau) $ coming from the limiting ray $ R_\infty $ we set 
$$ \widehat{f}(0,\tau,z) = (n-1)! 2^{n-1} \, \int_{\C^n} f^0(z-w) \frac{J_{n-1}(\sqrt{\tau} |w|)}{(\sqrt{\tau} |w|)^{n-1}} dw $$ 
where $ J_{n-1} $ is the Bessel function of order $(n-1).$ At $ (0,0,z) $ we let $ \widehat{f}(0,0,z) = \int_{\He^n} f(w,t) dw dt.$ 
As a subset of $ \R^2 , \Omega $ inherits the Euclidean metric and topology.\\

We define the normalised Strichartz Fourier transform by  $ \widetilde{f}(a,z) = \frac{k!(n-1)!}{(k+n-1)!} \widehat{f}(a,z) $ for $ a \in R_k.$ On the limiting ray we simply set $ \widetilde{f}(0,\tau,z) = \widehat{f}(0,\tau,z).$  It then follows that   $  \widetilde{f} $ is a continuous  function on $ \Omega  $ for $ f \in L^1(\He^n).$\\

{\bf Proof of Theorem \ref{riem-leb}:}  The only nontrivial part which requires a proof is  the continuity at $ (0,\tau,z) $  when a sequence  $ a_j \in \Omega $ converges to $ (0,\tau) $ running through different rays $ R_k.$ Thus we have $ (\lambda_j, (2k_j+n)|\lambda_j|) $ with $ \lambda_j \rightarrow 0 $ and $ (2k_j+n) |\lambda_j| \rightarrow \tau.$ We need to show that
$$ \lim_{j \rightarrow \infty} \frac{k_j! (n-1)!}{(k_j+n-1)!} \, f^{-\lambda_j} \ast_{-\lambda_j} \varphi_{k_j,\lambda_j}^{n-1}(z) = (n-1)! 2^{n-1} \, \int_{\C^n} f^0(z-w) \frac{J_{n-1}(\sqrt{\tau} |w|)}{(\sqrt{\tau} |w|)^{n-1}} dw.$$ 
This can be proved by using  asymptotic properties of the Laguerre functions $ \varphi_{k,\lambda}^{n-1}(z).$  From \cite[(8.22.4),\, p.193]{GS} we have
\begin{equation} \label{asymp-lag} \varphi_{k,\lambda}^{n-1}(z) = (n-1)! 2^{n-1} \frac{ J_{n-1}(\sqrt{(2k+n)|\lambda|}\,|z|)}{(\sqrt{(2k+n)|\lambda|}\,|z|)^{n-1}} + m_k( \sqrt{|\lambda|} |z|) 
\end{equation}
where the error term satisfies the uniform estimates 
\begin{equation}\label{error-esti} |m_k(t)| \leq C\, (2k+n)^{-(n-1)/2-3/4},\,\,\, 0 < t \leq b. \end{equation}   Thus we see that $ \widetilde{f}(a_j,z) $ is a sum of two terms of which the main term is given by
$$ (n-1)! 2^{n-1} \, \int_{\C^n} f^{\lambda_j}(z-w)\, \frac{ J_{n-1}(\sqrt{(2k_j+n)|\lambda_j|}\,|w|)}{(\sqrt{(2k_j+n)|\lambda_j|}\,|w|)^{n-1}}  \, e^{\frac{i}{2} \lambda_j \Im(z \cdot \bar{w})} \, dw.$$
As $ \lambda_j \rightarrow 0 $ and $ (2k_j+n) |\lambda_j| \rightarrow \tau,$ it is clear that the above converges to 
$$ (n-1)! 2^{n-1} \, \int_{\C^n} f^0(z-w) \frac{J_{n-1}(\sqrt{\tau} |w|)}{(\sqrt{\tau} |w|)^{n-1}} dw.$$ 
As $ \lambda_j $ remains bounded, the estimates on $ m_k(\sqrt{|\lambda|} \,|z-w|) $ shows that the error term goes to zero proving our claim. \\

In order to prove that $ \widetilde{f}(a,z) $ vanishes at infinity, we first observe that the  Strichartz Fourier transform of any $ f \in L^1(\He^n) $ satisfies the estimate
\begin{equation}\label{bound for L one} \sup_{(a,z) \in \Omega \times \C^n} |\widetilde{f}(a,z)| \leq c_{n,1}\, \|f\|_1 .\end{equation} This is a consequence of the well known fact that the Laguerre functions $ \varphi_{k,\lambda}^{n-1}(z) $ and the Bessel function $ J_{n-1}(t)$  satisfy the uniform estimates (see \cite[(8.22.4) p. 193, (1.71.6) p. 15]{GS} )
\begin{equation}\label{estimates} \frac{k!(n-1)!}{(k+n-1)!}\,  |\varphi_{k,\lambda}^{n-1}(z)| \leq 1,\,\,\,\,\,  |J_{n-1}(t)| \leq c_n \, t^{n-1}.\end{equation} 
Therefore, it is enough to prove that $ \widetilde{f}(a,z) $ vanishes at infinity whenever $ f $ is compactly supported. The case $ a = (0,\tau) \rightarrow \infty $ is easy to handle. In this case 
$$ \widehat{f}(0,\tau,z) = (n-1)!\, 2^{n-1} \int_{\C^n} f(z-w) \frac{ J_{n-1}(\sqrt{\tau}|w|)}{(\sqrt{\tau}|w|)^{n-1}} \, dw $$
is a constant multiple of the Hankel transform of the radial function 
$$ f_z(w) = \int_{U(n)} f(z-\sigma w) d\sigma .$$ 
Therefore, by appealing to Riemann-Lebesgue lemma for Hankel transforms we can conclude that $ \widehat{f}(0,\tau,z) $ vanishes at infinity.

Consider the case when $ a = (\lambda, (2k+n)|\lambda|) \in R_k $ goes to infinity. In case $ |\lambda| \rightarrow \infty $ we have
$$ \frac{k!(n-1)!}{(k+n-1)} |\widehat{f}(a,z)|  \leq \int_{\C^n} |f^\lambda(z)|\, dz $$ 
in view of \eqref{estimates}. The above certainly vanishes as $ |\lambda| \rightarrow \infty $ in view of Riemann-Lebesgue lemma for the Euclidean Fourier transform.  In case $ (2k_j+n)|\lambda_j \rightarrow \infty $ but $ \lambda_j $ remains bounded, the main term in $ \widetilde{f}(a_j,z) $ is bounded by a contant times
$$ \int_{\C^n} |f^{\lambda_j}(z-w)|\, \frac{ |J_{n-1}(\sqrt{(2k_j+n)|\lambda_j|}\,|w|)|}{(\sqrt{(2k_j+n)|\lambda_j|}\,|w|)^{n-1}}   \, dw.$$
Since $ |f^{\lambda_j}(w)| \leq |f^0(w)| $ is integrable, the estimate $ |J_{n-1}(t)| \leq c\, t^{n-1}\, (1+t)^{-1/2} $ shows that the main term goes to zero as $ j $ goes to infinity. The same is true of the error term in view of \eqref{error-esti} since $ f^0(w)$ is compactly supported. This completes the proof.\\

{\bf Proof of Theorem \ref{inv-plan}:} In view of the relation (see Corollary 2.3.4 in \cite{ST3})
$$ \varphi_{k,\lambda}^{n-1} \ast_\lambda \varphi_{k,\lambda}^{n-1}(z) = (2\pi)^n |\lambda|^{-n}\, \varphi_{k,\lambda}^{n-1}(z) $$ we see that the Fourier transform $  \widehat{f}(a,z)$ satisfies
\begin{equation}\label{necessary}  (2\pi)^{-n} |\lambda|^n\, \int_{\C^n}  \widehat{f}(a,w) \varphi_{k,\lambda}^{n-1}(z-w) e^{\frac{i}{2}\lambda \Im( z\cdot \bar{w})} dw =  f^{-\lambda} \ast_{-\lambda} \varphi_{k,\lambda}^{n-1}(z) =   \widehat{f}(a,z).
\end{equation}
Recalling the definition of $ e_{k,\lambda}^{n-1}(z,t) $ we can rewrite the above as
$$ (2\pi)^{-n} \, |\lambda|^n\, \int_{\C^n}  \widehat{f}(a,w) e_{k,\lambda}^{n-1}((w,0)^{-1}(z,t)) dw =  \, e^{i\lambda t} \, f^{-\lambda} \ast_{-\lambda} \varphi_{k,\lambda}^{n-1}(z).$$
Integrating the above over $ \Omega $ with respect to $ d\nu $ and recalling  \eqref{inv} we obtain the 
inversion formula
\begin{equation}\label{inversion} f(z,t) = \int_\Omega \int_{\C^n}  \widehat{f}(a,w) e_a((-w,0)(z,t))  dw \, d\nu(a).
\end{equation}
The Plancherel theorems for the special Hermite expansions \eqref{spl-her-plan} and the Euclidean Fourier transform gives us 
 the identity
\begin{equation}\label{plancherel} \int_{\He^n} |f(z,t)|^2 dz\, dt = c_n \,  \int_\Omega \int_{\C^n}  |\widehat{f}(a,w)|^2 dw \,d\nu(a).
\end{equation}
This completes the proof of Theorem \ref{inv-plan}.\\

{\bf Proof of Theorem \ref{unitary}:}  For a given $  F \in  L^2_0(\Omega \times \C^n, d\nu\, dw),$ we define
\begin{equation} f(z,t) = \int_\Omega \int_{\C^n}  F(a,w) e_a((-w,0)(z,t))  dw \, d\nu(a).
\end{equation}
As $  F \in  L^2_0(\Omega \times \C^n, d\nu\, dw),$ the function $ f $ defined by
$$ f(z,t) = (2\pi)^{-n-1} \, \int_{-\infty}^\infty \big( \sum_{k=0}^\infty  e^{i\lambda t} F(a,z) \big) |\lambda|^{n} d\lambda $$
belongs to $L^2(\He^n) .$  But then  
$ f^{-\lambda}(z) = (2\pi)^{-n} |\lambda|^n \big( \sum_{k=0}^\infty  F(a,z) \big), $ which gives, in view of the
  condition \eqref{necessary} and the orthogonality relation $ \varphi_{k,\lambda}^{n-1} \ast_{-\lambda} \varphi_{j,\lambda}^{n-1} = 0,\,  j \neq k$ we obtain
 $$   \widehat{f}(a,z)= f^{-\lambda} \ast_{-\lambda} \varphi_{k,\lambda}^{n-1}(z) = F(a,z) $$
 which completes the proof.\\

{\bf Proof of Theorem \ref{haus-young}:} 
 Recall that in the course of the proof of Theorem \ref{riem-leb} we have verified the estimate \eqref{bound for L one}, namely
$$ \sup_{(a,z) \in \Omega \times \C^n} |\widetilde{f}(a,z)| \leq c_{n,1}\, \|f\|_1 .$$
This simply means that $ f \rightarrow \widetilde{f} $ is bounded from $ L^1(\He^n) $ into $ L^\infty(\Omega \times \C^n, d\nu_1(a)\, dz ) $ where  
$$ \int_{\Omega} \varphi(a) d\nu_1(a) = (2\pi)^{-2n-1}\, \int_{-\infty}^\infty \Big( \sum_{k=0}^\infty \varphi(\lambda, (2k+n)|\lambda|)\Big) d\lambda.$$
Restating the Plancherel  theorem for $ \widetilde{f} $ we also have
$$\Big(  \int_\Omega \int_{\C^n}  |\widetilde{f}(a,w)|^2 dw \,d\nu_2(a) \Big)^{1/2} =  \Big( \int_{\He^n} |f(z,t)|^2 dz\, dt \Big)^{1/2} $$
with another measure  $ d\nu_2(a)$  on $ \Omega $ defined by
$$ \int_{\Omega} \varphi(a) d\nu_2(a) = \int_{-\infty}^\infty \Big( \sum_{k=0}^\infty \Big(\frac{k+n-1)!}{k!(n-1)!} \Big)^2 \, \varphi(\lambda, (2k+n)|\lambda|)\Big) |\lambda|^{2n} d\lambda.$$  
Using interpolation theorem with change of measures  \cite{SW} we obtain the Hausdorff-Young inequality
$$ \Big( \int_\Omega \int_{\C^n}  |\widetilde{f}(a,w)|^{p^\prime} dw \,d\nu_p(a) \Big)^{1/p^\prime} \leq  c_{n,p}\, \Big(\int_{\He^n} |f(z,t)|^p dz\, dt \Big)^{1/p}$$
for all $ f \in L^p(\He^n), 1 \leq p \leq 2 $ for some measure $ \nu_p(a) $ on $\Omega.$ The measure $ \nu_p(a) $ can be constructed explicitly.
\\

\begin{rem} Since  $ f \rightarrow \widetilde{f} $ is also bounded from $ L^1(\He^n) $ into $ L^\infty(\Omega \times \C^n, d\nu_2(a)\, dz ), $ interpolation without change of measures gives the inequality
$$ \Big( \int_\Omega \int_{\C^n}  |\widetilde{f}(a,w)|^{p^\prime} dw \,d\nu_2(a) \Big)^{1/p^\prime} \leq  c_{n,p}^\prime\, \Big(\int_{\He^n} |f(z,t)|^p dz\, dt \Big)^{1/p}.$$
However, the result proved above with change of measures is sharper than this inequality.
\end{rem}

\section{ Strichartz Fourier transform vs. Gelfand transform} We now investigate further properties of the Fourier transform $ f \rightarrow \widehat{f} $ which justifies our claim that this transform is the analogue of the Helgason Fourier transform on Riemannian symmetric spaces of noncompact type. If $ f $ is a $ K$-biinvariant function on $ X = G/K$ it is known that the Helgason Fourier transform $ \widetilde{f}(\lambda,b) $ is independent of $ b $ and reduces to the Jacobi transform. In fact if $ f(g) = f(k a_r k^\prime) = f_0(r) $ then 
$ \widetilde{f}(\lambda,b) = J_{\alpha,\beta}f_0(\lambda) Y_0(b) $ where $ Y_0(b) = 1 $ is the unique $K$-fixed vector associated to the representation $ \pi_\lambda$ and
$$ J_{\alpha,\beta}f_0(\lambda) = \int_0^\infty f_0(r) \varphi_\lambda^{\alpha,\beta}(r) w_{\alpha,\beta}(r) dr $$ 
where $ \varphi_\lambda^{\alpha,\beta}$ is the Jacobi function of type $ (\alpha,\beta) $ which depends on the symmetric space. For the Helgason Fourier transform, it is not true that $ \widetilde{\big(f \ast g \big)}(\lambda,b) = \widetilde{f}(\lambda,b)\widetilde{g}(\lambda,b).$ However if $ g $ is $K$-biinvariant, then it is indeed true and we have
$$   \widetilde{\big(f \ast g \big)}(\lambda,b) = J_{\alpha,\beta}g_0(\lambda)\, \widetilde{f}(\lambda,b),\,\,\, g(k a_r k^\prime) = g_0(r).$$ We have a similar situation for the Strichartz  Fourier transform on the Heisenberg group.

\subsection{Gelfand transform on the Heisenberg group} In the case of $ \He^n$, the role of $K$-biinvariant functions is played by $U(n)$-invariant functions of $ \He^n,$ also known as radial functions. Note that when $ f $ is such a function, then $ f(z,t) = f_0(|z|,t) $ for a unique function on $ \R^+ \times \R $ and we know that the special Hermite expansion of the radial function $ f^\lambda(z) $ reduces to the Laguerre expansion, see \cite[(2.4.5),p. 61]{ST3} . Thus
$$ f^\lambda \ast_\lambda \varphi_{k,\lambda}^{n-1}(z) = R_k^{n-1}(\lambda, f) \varphi_{k,\lambda}^{n-1}(z) $$
where $R_k^{n-1}(\lambda, f) $ is the $k$-th Laguerre coefficient of $ f^\lambda $ given by
$$R_k^{n-1}(\lambda, f)   = (2\pi)^n\, 2^{-n+1}\, \frac{k!}{(k+n-1)!} \, \int_0^\infty f_0^\lambda(r) \varphi_{k,\lambda}^{n-1}(r) \, r^{2n-1} dr.$$
Since the surface measure of $ S^{2n-1} $ is  $ 2\, \frac{ \pi^n}{\Gamma(n)} $ we can write the above in the form
$$R_k^{n-1}(\lambda, f)   = \frac{k! (n-1)!}{(k+n-1)!} \, \int_{\C^n} f^\lambda(z) \varphi_{k,\lambda}^{n-1}(z) \, dz.$$
Thus we see that if $ f $ is a radial function on $ \He^n, $ then 
\begin{equation}\label{fouri-radi} \widehat{f}(a,z) = R_k^{n-1}(-\lambda, f)  e_{k,\lambda}^{n-1}(z,0) 
\end{equation}
so that $ \widehat{f}(a,z)$ is proportional to the unique $U(n)$-fixed vector 
for the representation $ \rho_k^\lambda.$  For $ f , g \in L^1(\He^n) $ we also have
$$ \widehat{f\ast g}(a,z) = R_k^{n-1}(-\lambda, g_0) \, \widehat{f}(a,z).$$
As in the case of  the Helgason Fourier transform, a more general result, known as the Hecke-Bochner formula is true.\\

  Though the algebra $ L^1(\He^n) $ under convolution is noncommutative, the subalgebra $ L^1(\He^n/U(n)) $ consisting of radial functions is commutative. This follows from the fact that 
for any radial function $ f $ the special Hermite expansion of $ f^\lambda $ reduces to the Laguerre expansion as proved above:
\begin{equation}\label{lag-exp} f^\lambda(z) = (2\pi)^{-n}\, |\lambda|^n \, \sum_{k=0}^\infty R_k^{n-1}(\lambda, f)\, \varphi_{k,\lambda}^{n-1}(z).
\end{equation} 
As elements of $ L^1(\He^n/U(n)) $ are precisely the $ U(n)$-biinvariant functions on the Heisenberg motion group, this simply means that $ (G_n, U(n)) $ is a Gelfand pair. The multiplicative linear functionals on this algebra are given by bounded spherical functions which come in two families as shown in \cite{BJR}. These are given by
\begin{equation}\label{spheri-fun} \widetilde{e}_a(z,t) = \frac{k! (n-1)!}{(k+n-1)!}\, e_{k,\lambda}^{n-1}(z,t),\,\,  \chi_\tau(z,t) = (n-1)! 2^{n-1} \frac{J_{n-1}(\sqrt{\tau}|z|)}{(\sqrt{\tau}|z|)^{n-1}}, \, \tau  \geq 0.
\end{equation}
Thus we see that the Gelfand spectrum of the algebra $ L^1(\He^n/U(n)) $ is precisely the Heisenberg fan $ \Omega.$ The Gelfand transform on $ L^1(\He^n/U(n)) $ is defined as the map which takes $ f $ into $ \mathcal{G}f $ given by
$$ \mathcal{G}f(a) =  \int_{\He^n} f(z,t) \, \widetilde{e}_a(z,t)\, dz\, dt ,\,\, \, \mathcal{G}f(0, \tau) = \int_{\He^n} f(z,t) \,\chi_\tau(z,t)\, dz\,dt.$$
Thus the relation \eqref{fouri-radi} reads as 
\begin{equation}\label{stri-gel} \widehat{f}(a,z) =  \mathcal{G}f(a)\, e_{k,\lambda}^{n-1}(z,0),\,\,  f \in L^1(\He^n/U(n)).
\end{equation}
The same relation holds also for $ a = (0,\tau) $ as can be easily verified.\\

From the expansion \eqref{lag-exp} for  radial functions we see that their Fourier transforms are functions of the Hermite operator:
$$ \widehat{f}(\lambda) = \sum_{k=0}^\infty R_k(\lambda,f)\, P_k(\lambda) =  \sum_{k=0}^\infty  \mathcal{G}f(a)\, P_k(\lambda).$$
This relation allows us to study the Fourier transform of radial functions in terms of their Gelfand transforms, also  called spherical Fourier transforms for obvious reasons. As the Gelfand transform is scalar valued, it has been used as an alternate for the group Fourier transform by several authors. For example, in the work \cite{ADR1} the authors have studied the image of radial Schwartz functions under the spherical Fourier transform.  In what follows, we will further explore the connection between the Strichartz Fourier transform and  Gelfand transforms.

\subsection{Hecke-Bochner formula for the Strichartz Fourier transform} We begin by recalling briefly some basic facts about bigraded spherical harmonics, referring to \cite[Section 2.5]{ST3} for details. Given a pair $ (p,q) $ of non-negative integers, we say that a polynomial $ P(z, \bar{z}) $ on $ \C^n $ is a bigraded solid harmonic if it is harmonic and has the form 
$$ P(z,\bar{z}) = \sum_{|\alpha| =p} \sum_{|\beta|=q} c_{\alpha, \beta} \, z^\alpha\, \bar{z}^\beta.$$
They are uniquely determined by their restrictions to the unit sphere $ S^{2n-1}.$ We denote by $ \mathcal{S}_{p,q} $ the space of all bigraded spherical harmonics, i.e. restrictions of bigraded solid harmonics. It is then known that
$ L^2(S^{2n-1}) = \bigoplus_{p,q} \mathcal{S}_{p,q}. $ Each $ \mathcal{S}_{p,q}$ is finite dimensional and by choosing an orthonormal basis $ S_{p,q}^j, j =1,2,..., d(p,q) $ for each $ (p,q) $ we obtain an orthonormals basis for $ L^2(S^{2n-1}).$\\

Let us return to the Euclidean Fourier transform \eqref{e-fouri} for a moment and consider the integral
$$ \int_{S^{n-1}} \widehat{f}(\lambda, \omega)\, Y_m(\omega)\, d\omega $$
 where $ Y_m $ is a (standard) spherical harmonic of degree  $ m.$  Then it is known that (see for example \cite{SW1})
 \begin{equation}\label{e-fouri-harm}  \lambda^{-m}\, \int_{S^{n-1}} \widehat{f}(\lambda, \omega)\, Y_m(\omega)\, d\omega  = c_{n,m}  \int_0^\infty  r^{-m} f_m(r) \, \frac{J_{n/2+m-1}(\lambda r)}{(\lambda r)^{n/2+m-1}} r^{n+2m-1} dr 
 \end{equation}
where $ f_m(r) $ is the spherical harmonic coefficient $ (f(r\cdot),Y_m)_{L^2(S^{n-1})}.$  Suppose $ g $ is a radial function on $ \R^n$ and $ f(x) = g(x) P_m(x) $ where $ P_m $ is a solid harmonic of degree $m.$ Then from the above formula,
we easily see that the spherical harmonic expansion of $ \widehat{f}(\lambda,\omega) $ has only one term and hence
\begin{equation}\label{e-hecke} \widehat{f}(\lambda,\omega) = c_{n,m} \, P_m(\lambda \omega) \, \int_0^\infty  g(r) \, \frac{J_{n/2+m-1}(\lambda r)}{(\lambda r)^{n/2+m-1}} r^{n+2m-1} dr .
\end{equation}
We also observe that if we consider $ g $ as a radial function on $ \R^{n+2m}$ then the above can be rewritten as $ \mathcal{F}_n(gP_m) = c_{n,m} P_m \, \mathcal{F}_{n+2m}(g)$ where $ \mathcal{F}_k $ stands for the Fourier transform on $ \R^k.$
These formulas are known as Hecke-Bochner identities in the literature.\\

There is a Hecke-Bochner formula for the Helgason Fourier transform too, but we do not intend to state it here as it requires  quite a bit of preparation, see \cite{SH1,SH2}. Given below is the Hecke-Bochner formula for the Strichartz Fourier transform which is the exact analogue of \eqref{e-fouri-harm}. In what follows we use the following convention. A radial function $ g $ on $ \He^n $  will be simultaneously considered as a radial  function on any $ \He^m.$  When $ k$ is a negative integer we set $ \varphi_{k,\lambda}^m(z)  = 0$ for any $ m $ and $ z \in \C^m.$  Moreover, the Gelfand tranform for the algebra $ L^1(\He^m/U(m)) $ will be denoted by $ \mathcal{G}_m.$ With these notations we have\\

\begin{thm} For any $ f \in L^1(\He^n) $ and $ S_{p,q} \in \mathcal{S}_{p,q} ,$  let $ f_{p,q}(r,t) = (f(r\cdot,t),S_{p,q})_{L^2(S^{2n-1})} $ stand for the spherical harmonic coefficient of $ f(r\omega, t).$ For any $ a \in R_k, \lambda > 0 $ we have the following:
$$ \int_{S^{2n-1}} \widehat{f}(a, r\omega) \, S_{p,q}(\omega)\, d\omega = (2\pi)^{-p-q} (|\lambda| r)^{p+q} \, \mathcal{G}_{n+p+q}(g_{p,q})(a(p,q)) \, e_{k-q,\lambda}^{n+p+q-1}(r,0)$$
where $ g_{p,q}(z,t) = |z|^{-p-q}\, f_{p,q}(|z|,t) $ and $ a(p,q) = (\lambda, (2k+p-q+n)|\lambda|).$ When $ \lambda <0$ we have
$$ \int_{S^{2n-1}} \widehat{f}(a, r\omega) \, S_{p,q}(\omega)\, d\omega = (2\pi)^{-p-q} (|\lambda| r)^{p+q} \, \mathcal{G}_{n+p+q}(g_{p,q})(a(q,p)) \, e_{k-p,\lambda}^{n+p+q-1}(r,0).$$
\end{thm}
\begin{proof} The proof is indeed a rewriting of  \cite[Theorem 2.6.1]{ST3} where the following result is proved. Let $ f \in L^1(\He^n) $ be of the form $ f(z,t) = g(z,t) P(z) $ where $ g $ is radial and $ P $ is a bigraded solid harmonic of degree $(p,q).$
Then for any $ \lambda > 0 $ we have
\begin{equation}\label{spl-her-heck} f^\lambda \ast_\lambda \varphi_{k,\lambda}^{n-1}(z) = (2\pi)^{-p-q} \lambda^{p+q}\, P(z) \, g^\lambda \ast_\lambda \varphi_{k-p,\lambda}^{n+p+q-1}(z) 
\end{equation}
where the convolution on the right hand side is taken over $ \C^{n+p+q}.$ There is a similar formula when $ \lambda < 0 $ where the roles of $ p $ and $q$  are interchanged.  This result, stated in terms of Weyl transform, is due to Geller \cite{DG}. It is enough to prove the result when $ S_{p,q} = S_{p,q}^j $ is a member of the orthonormal basis for $ \mathcal{S}_{p,q} $ which we have described earlier.

Expanding $ f^\lambda(z) $ in terms of $ S_{p,q}^j $ and recalling the definitions of  $ f_{p,q} $ and $ g_{p,q}$ we have
$$ f^\lambda(z) = \sum_{p,q} \sum_{j=1}^{d(p,q)} P_{p,q}^j(z)\, (g_{p,q}^j)^\lambda(z).$$
As $ \widehat{f}(a,z) = f^{-\lambda}\ast_{-\lambda} \varphi_{k,\lambda}^{n-1}(z) ,$ assuming $ \lambda < 0 $ (so that $ -\lambda >0$), the above formula gives
$$ \widehat{f}(a,z) = \sum_{p,q} (2\pi)^{-p-q} |\lambda|^{p+q}\, \Big(\sum_{j=1}^{d(p,q)} P_{p,q}^j(z)\, (g_{p,q}^j)^{|\lambda|} \ast_{|\lambda|} \varphi_{k-p,\lambda}^{n+p+q-1}(z)\Big).$$
As $ g_{p,q}^j(z,t) $ is radial on $ \He^{n+p+q} $ using \eqref{stri-gel} we can rewrite the above expansion as
$$ \widehat{f}(a,z) = \sum_{p,q} (2\pi)^{-p-q} |\lambda|^{p+q}\, \Big(\sum_{j=1}^{d(p,q)} P_{p,q}^j(z)\, \mathcal{G}_{n+p+q}(g_{p,q}^j)(a(q,p)) \Big)\, e_{k-p,\lambda}^{n+p+q-1}(z,0) .$$
By calculating the spherical harmonic coefficients of $  \widehat{f}(a,r\omega) $ we obtain the stated formula.\\
\end{proof}

The identity \eqref{spl-her-heck} which was used in the proof of the  above theorem can be restated as the following Hecke-Bochner formula for the Strichartz Fourier transfiorm.

\begin{cor} Suppose $ f \in L^1(\He^n) $ is of the form $ f(z,t) = P(z) g(z,t) $ where $ g $ is radial and $ P $ is a solid harmonic of degree $(p,q).$ Then for any $ \lambda >0,$
$$ \widehat{f}(a,z) = (2\pi)^{-p-q} |\lambda|^{p+q} \, P(z)\, \mathcal{G}_{n+p+q}g(a(p,q)) \, e_{k-q,\lambda}^{n+p+q-1}(z,0).$$
A similar formula holds for $ \lambda <0.$\\
\end{cor}

We also have a representation theoretic interpretation of the Hecke-Bochner formula. 
The spherical harmonic expansion of a function $ f(z,t) $ on $ \He^n$ can be written as 
\begin{equation}\label{spher-harm-exp} f(z,t) = \sum_{(p,q) \in \mathbb{N}^2}  \sum_{j=1}^{d(p,q)}g_{(p,q),j}(r,t)\, P_{p,q}^j(z),\,\,\, z = r\omega \end{equation}
where $ P_{p,q}^j $ is the solid harmonic corresponding to $ S_{p,q}^j.$ For each pair $(p,q) $ the space $\mathcal{S}_{p,q} $ supports an irreducible unitary representation $ R_{(p,q)} $ of the unitary group $ K= U(n).$ The action of $ R_{(p,q)}$ on $\mathcal{S}_{p,q} $ is given by  $ R_{(p,q)}(\sigma) S(\omega) = S(\sigma^{-1}\omega), \sigma \in K.$ Let $ M  $ be the isotropic subgroup of $ K $ which fixes the coordinate vector $ e_1 \in \C^n $  which can be indentified with $ U(n-1).$ It is known that each $ R_{(p,q)} $ has a unique $M$-fixed vector in $ \mathcal{S}_{p,q}.$ Such representations are called class one representations and it is known that the any such irreducible unitary representation of $ K $ is unitarily equivalent to one and only one of $ R_{(p,q)}.$ Let $ \widehat{K}_0 $ stand for the set of all equivalence classes of such representations of $ K.$ We can rewrite \eqref{spher-harm-exp} as
\begin{equation}\label{spher-harm-exp-one}
f(z,t) = \sum_{\delta \in \widehat{K}_0}  f_\delta(z,t),\,\,\,\, \,\, \,\,\, f_\delta(z,t) =  \sum_{j=1}^{d(p,q)}g_{(p,q),j}(r,t)\, P_{p,q}^j(z), \,\,\,  \delta = (p,q).
\end{equation}
We can view the functions $ f_\delta $ as radial functions on $ \He^n $ taking values in the finite dimensional Hilbert space $ \mathcal{S}_\delta.$\\

Let $ L^2(\He^n/U(n), \mathcal{H}_\delta) $ stand for the space of all such radial functions taking values in $ \mathcal{H}_\delta$ so that 
\begin{equation}\label{direct-sum} L^2(\He^n) = \bigoplus_{\delta \in \widehat{K}_0}  L^2(\He^n/U(n), \mathcal{S}_\delta).\end{equation}
It then follows that each of the spaces $ L^2(\He^n/U(n), \mathcal{H}_\delta) $ is invariant under the Strichartz Fourier transform. Indeed, by the Hecke-Bochner formula it is clear that, for $ \lambda >0$ 
\begin{equation}\label{fouri-delta} 
 \widehat{f_\delta}(a,z) = (2\pi)^{-p-q} |\lambda|^{p+q} \, \Big(  \sum_{j=1}^{d(p,q)} P_{p,q}^j(z)\, \mathcal{G}_{n+p+q}(g_{\delta,j})(a(p,q))\Big)  \, \varphi_{k-q,\lambda}^{n+p+q-1}(z).
\end{equation}
Thus we can consider the Strichartz Fourier transform as a family of Gelfand transforms $ \mathcal{G}_\delta $ indexed by the class one representations  $ \delta $ of  the pair $(U(n),U(n-1))$.  If $ f_\delta $ is thus identified with the vector $ ( g_{\delta,j})_{j=1}^{d(\delta)},$ then $  \widehat{f_\delta}(a,z)$ is given by the vector with components 
\begin{equation}\label{fouri-delta-compo} (2\pi)^{-p-q} |\lambda|^{p+q} \, r^{p+q}\,  \mathcal{G}_{n+p+q}(g_{\delta,j})(a(p,q))  \, \varphi_{k-q,\lambda}^{n+p+q-1}(r).
\end{equation} 
We will make use of this view in the next section where we study the image of the Schwartz space under the Strichartz Fourier transform.\\

\subsection{The image of $\mathcal{S}(\He^n)$ under the Strichartz Fourier transform} Let $ \mathcal{S}(\He^n) $ stand for the space of all Schwartz class functions on the Heisenberg group. As the underlying manifold of $ \He^n $ is just $ \R^{2n+1} $ this space is nothing but $ \mathcal{S}(\R^{2n+1}).$ In his pioneering work Geller \cite{DG} has proved a characterisation of the image of $ \mathcal{S}(\He^n) $ under the group Fourier transform in terms of certain asymptotic series. In 1998, Benson et al. \cite{BJR} studied the image of Schwartz functions on the Heisenberg group under the spherical Fourier transform associated to Gelfand pairs. In particular, for radial functions on $ \He^n$ they have described the image as a space of rapidly decreasing functions on the Gelfand spectrum identified with the Heisenberg fan $ \Omega.$  In their work, rapidly decreasing functions on $ \Omega $ are defined in terms of certain derivatives and finite difference operators.\\

In a series of papers \cite{ADR1,ADR2, ADR3, ADR4}, Astengo et al have studied the problem of characterizing the image of $ \mathcal{S}(\He^n) $ under the Fourier transform. By using multiple Fourier series, they have reduced the problem to the characterization of the image of polyradial Schwartz functions under the Gelfand transform (spherical Fourier transform). See the survey \cite{FR}  for a readable account of these works with connections to spectral multipliers. In particular, for the class $ \mathcal{S}_K(\He^n) $ of radial functions (recall that $ K = U(n)$) they have proved the following elegant result.\\

\begin{thm}\cite[Astengo et al.]{ADR1}\label{fourier-image} Let $ \mathcal{S}(\Omega) $ stand for the space of restrictions of Schwartz functions on $ \R^2 $ to $ \Omega$ equipped with the quotient topology $ \mathcal{S}(\R^2)/ \{ f: f|_\Omega =0 \}.$ The Gelfand transform $ \mathcal{G}_n $ is a topological isomorphism between $\mathcal{S}_K(\He^n) $ and $\mathcal{S}(\Omega) .$\\
\end{thm} 

Recall that the  components $ g_{\delta,j} $ of $ f_\delta $ are defined by $ r^{-p-q} f_{\delta,j}(r,t)$ where
$$ f_{\delta,j}(r,t) = \int_{S^{2n-1}} f(r\omega,t) S_{\delta,j}(\omega)\, d\omega.$$
It is clear that when $ f \in \mathcal{S}(\He^n)$ these functions $ g_{\delta,j} $ considered as radial functions on $ \He^{n+p+q} $ are Schwartz functions. For the sake of brevity let us use the following notations: we write $ \He^\delta $ in place of $ \He^{n+p+q}$ and $ U(\delta) $ in place of $ U(n+p+q)$ so that $ \He^n$ and $ U(n) $ correspond to the trivial representation of $ U(n).$ Following the previous authors, let  $ \mathcal{S}_{U(\delta)}(\He^\delta, \mathcal{S}_\delta) $ stand for the Schwartz space of $ U(\delta) $ invariant functions on $ \He^\delta$ taking values in $ \mathcal{S}_\delta.$ The Gelfand transform for the pair $ (\He^\delta,U(\delta)) $ will be denoted by $ \mathcal{G}_\delta.$  We then observe that for Schwartz class functions the decomposition \eqref{direct-sum} takes the form
$$ \mathcal{S}(\He^n) =  \bigoplus_{\delta \in \widehat{K}_0} \mathcal{S}_{U(\delta)}(\He^\delta, \mathcal{S}_\delta).$$
We also use the notation 
$ \mathcal{S}(\R^2,\mathcal{S}_\delta) $ for the space of $ \mathcal{S}_\delta $ valued Schwartz functions on $ \R^2.$ We define $ \mathcal{S}(\Omega,\mathcal{S}_\delta) $ as in the scalar valued case and take their orthogonal sum to define
$$ \widehat{\mathcal{S}}(\Omega) = \bigoplus_{\delta \in \widehat{K}_0} \mathcal{S}(\Omega,\mathcal{S}_\delta) .$$
With these notations we can restate Theorem \ref{fourier-image} in the following form.\\

\begin{thm}\label{fourier-image-new} The Strichartz Fourier transform is an isomorphism between $ \mathcal{S}(\He^n)$ and $ \widehat{\mathcal{S}}(\Omega).$
\end{thm} 

Even though this result is a consequence of Theorem \ref{fourier-image}, due to the various identifications we have made, the following explanations are in order. Given $ f \in \mathcal{S}(\He^n) $ and $ f_\delta $ defined as in \eqref{spher-harm-exp-one} the function 
$$ \widetilde{f}_\delta(z,t)  =  \sum_{j=1}^{d(\delta)}g_{\delta,j}(z,t)\, S_{\delta,j},\,\,\,\,  (z,t) \in \He^\delta $$
is an element of the space $ \mathcal{S}_{U(\delta)}(\He^\delta, \mathcal{S}_\delta).$  The Gelfand transform $ \mathcal{G}_\delta \widetilde{f}_\delta $ is a function on the Heisenberg fan $ \Omega_\delta $ for the pair $(\He^\delta,U(\delta)) $ which is a proper subset of $ \Omega.$ By Theorem \ref{fourier-image}  there exists $ m_\delta \in \mathcal{S}(\Omega,\mathcal{S}_\delta) $ such that $ \mathcal{G}_\delta \widetilde{f}_\delta $ is the restriction of $ m_\delta $ to $ \Omega_\delta.$ 
If $ m_{\delta,j} $ are the components of $ m_\delta$ then we have the relation $ \mathcal{G}_\delta \widetilde{f}_\delta = \sum_{j=1}^{d(\delta)}m_{\delta,j}\, S_{\delta,j}$ and the Strichartz Fourier transform of $ f_\delta $ considered as a function on $ \He^n $ is given by
$$ \widehat{f}_\delta(a,z)  = (2\pi)^{-p-q} |\lambda|^{p+q} \, \Big(  \sum_{j=1}^{d(\delta)} P_{\delta}^j(z)\, m_{\delta,j}(a(p,q))\Big)  \, \varphi_{k-q,\lambda}^{n+p+q-1}(z)$$
for $ \lambda > 0 $ with a similar expression for $ \lambda <0.$ Thus the correspondence alluded to in the statement of the theorem is the one given by $\widehat{f}_\delta \rightarrow m_\delta.$ Conversely, given $ m_\delta \in \mathcal{S}(\Omega,\mathcal{S}_\delta) $ we can define $ g_{\delta,j} $ by applying $ \mathcal{G}_\delta^{-1} $ and the function $ f_\delta $ is then defined as in \eqref{spher-harm-exp-one}. It is a routine matter to check that the resulting $ f $ is Schwartz.\\

\section{ Group Fourier transform vs. Strichartz Fourier transform} 
In this section we investigate the relation between group Fourier transform and the Strichartz Fourier transform on $ \He^n.$ Though the former is defined in terms of the Schr\"odinger representations $ \pi_\lambda $ and the latter in terms of $ \rho_k^\lambda $ there is a direct connection between $ \widehat{f}(\lambda) $ and $ \widehat{f}(a,z).$ We will show that this relation allows us to recast some theorems  for the group Fourier transform in terms of the Strichartz Fourier transform.\\

Combining \eqref{projection} and \eqref{wplan} we obtain the following relation between the Strichartz and the group Fourier transforms:
$$ \int_{\C^n} |\widehat{f}(a,z)|^2 dz = (2\pi)^n\, |\lambda|^{-n} \, \|\widehat{f}(\lambda) P_k(\lambda)\|_{HS}^2.$$
As an application of this relation, let us rewrite theorems of Hardy and Ingham for the Heisenberg group in terms of the Strichartz Fourier transform.\\

Let  $ p_a $ stand for the heat kernel associated to the sublaplacian $ \mathcal{L} $ on $ \He^n.$ The Heisenberg analogue of the classical Hardy's uncertainty principle reads as follows (see \cite[Theorem 2.9.2, p. 89]{ST3}): \\
 
{\it For a nontrivial function $ f \in L^1(\He^n) $ the condtions
$$ |f(z,t)| \leq c\, p_a(z,t),\,\,\, \widehat{f}(\lambda)^\ast \widehat{f}(\lambda) \leq c \, \widehat{p_{2b}}(\lambda) $$
cannot hold  simultaneously unless $ a \geq b.$} We can now restate this result in the following form.

\begin{thm}\label{hardy} Suppose $f \in L^1(\He^n) $ satisfies the conditions
$$ |f(z,t)| \leq c\, p_a(z,t),\,\,  |\lambda|^n \int_{\C^n} |\widehat{f}(a,z)|^2 dz \leq c\, \frac{(k+n-1)!}{k!(n-1)!}\, e^{-2b(2k+n)|\lambda|} $$
for every $ a \in R_k.$ Then $ f = 0 $ whenever $ a< b.$
\end{thm}

The decay condition on $ \widehat{f}(a,z) $ comes from the Hardy condition $ \widehat{f}(\lambda)^\ast \widehat{f}(\lambda) \leq c \, \widehat{p_{2b}}(\lambda).$ In fact, it is well known that $ \widehat{p_b}(\lambda) = e^{-bH(\lambda)} $ is the semigroup generated by the Hermite operator and an easy calculation using the Hermite basis shows that 
\begin{equation}\label{hardy-mod} \| \widehat{f}(\lambda) P_k(\lambda)\|_{HS}^2 \leq c\, \frac{(k+n-1)!}{k!(n-1)!}\, e^{-2b(2k+n)|\lambda|} \end{equation}
under the Hardy condition on $ \widehat{f}(\lambda).$ An examination of the proof of Hardy's theorem presented in \cite[Theorem 2.9.2]{ST3}, shows that the result holds under the weaker assumption \eqref{hardy-mod}. We remark that other more refined versions of Hardy's theorem can also be stated in terms of $ \widehat{f}(a,z).$\\

Ingham's uncertainty principle is another theorem that has received considerable attention in recent years. For functions on $ \R $ Ingham proved in 1934 the following result:\\

{\it Let $ \theta $ be a nonnegative even function on $ \R $ which decreases to $ 0 $ at infinity. There exists a compactly supported function $ f $ on $ \R$ whose Fourier transform satisfies the decay condition $$ |\widehat{f}(y)| \leq c\, e^{-|y| \, \theta(y)} $$ if and only if $ \int_1^\infty \theta(t) t^{-1} dt < \infty.$}\\

Recently the following analogue of Ingham's theorem for the group Fourier transform has been proved in \cite{BGST}:\\

{\it Let $ \Theta $ be a nonnegative even function on $ \R $ which decreases to $ 0 $ at infinity. There exists a compactly supported function $ f $ on $ \He^n$ whose Fourier transform satisfies the decay condition 
$$ |\widehat{f}(\lambda)| \leq c\, e^{-\sqrt{H(\lambda)}\, \Theta(\sqrt{H(\lambda)})} $$ if and only if $ \int_1^\infty \Theta(t) t^{-1} dt < \infty.$}\\

As before using the relation between $ \widehat{f}(\lambda) $ and $ \widehat{f}(a,z) $ we can restate the Ingham's theorem in the following form.

\begin{thm}\label{hardy} Let $ \Theta $ be a nonnegative even function on $ \R $ which decreases to $ 0 $ at infinity. There exists a compactly supported function $ f $ on $ \He^n$ whose Fourier transform satisfies the decay condition 
$$   |\lambda|^n \int_{\C^n} |\widehat{f}(a,z)|^2 dz \leq c\, \frac{(k+n-1)!}{k!(n-1)!}\, e^{-\sqrt{(2k+n)|\lambda|}\,\Theta(\sqrt{(2k+n)|\lambda|})} $$
 if and only if $ \int_1^\infty \Theta(t) t^{-1} dt < \infty.$
\end{thm}

The proof of Ingham's theorem given in  \cite{BGST}  can be easily modified to prove this version. We can also prove other refined versions stated and proved elsewhere. For more on Ingham's theorem and its close relatives we refer to \cite{GT} and the references therein.

\section*{Acknowledgments}


\end{document}